\documentclass{amsart}

\usepackage{amssymb}

\usepackage[dvips]{graphicx}

\usepackage{color}

\theoremstyle{plain}

\newtheorem{mainthm}{Theorem}

\swapnumbers
\newtheorem{theorem}{Theorem}[section]

\newtheorem{corollary}[theorem]{Corollary}
\newtheorem{lemma}[theorem]{Lemma}

\theoremstyle{definition}
\newtheorem*{definitions}{Definitions}
\newtheorem*{example}{Example}

\newtheorem{q}[theorem]{Question}

\newcommand{\Z}{\mathbb{Z}}
\newcommand{\N}{\mathbb{N}}

\newcommand{\eps}{\varepsilon}

\newcommand{\diam}{\operatorname{diam}}

\begin{document}

\begin{abstract}
We prove that the two-sided limit shadowing property is among the strong\-est known notions of pseudo-orbit tracing. It implies shadowing, average shadowing, asymptotic average shadowing and specification properties. We also introduce a weaker notion that is called two-sided limit shadowing with a gap and prove that it implies shadowing and transitivity. We show that those two properties allow to characterize topological transitivity and mixing in a class of expansive homeomorphisms and hence they characterize transitive (mixing) shifts of finite type.
\end{abstract}

\title[The two-sided limit shadowing property]{On homeomorphisms with the two-sided limit shadowing property}

\author[Bernardo Carvalho]{Bernardo Carvalho}
\author[Dominik Kwietniak] {Dominik Kwietniak}

\date{\today}

\thanks{This paper was partially supported by CNPq (Brazil). Research of Dominik Kwietniak was supported by the Polish Ministry of Science and Higher Education grant Iuventus Plus no. IP 2011028771. This research was initiated when the second author was visiting the Federal University of Rio de Janeiro. It was also conducted during a visit of the first author to the Jagiellonian Univeristy in Krakow. The financial support and kind hospitality of these institutions are gratefully acknowledged.}

\address[Bernardo Carvalho]{Universidade Federal do Rio de Janeiro - UFRJ}
\email{bmcarvalho06@gmail.com}
\address[Dominik Kwietniak]{Faculty of Mathematics and Computer Science, Jagiellonian Univesrity in Krak\'{o}w}
\email{dominik.kwietniak@uj.edu.pl}
\urladdr{www.im.uj.edu.pl/DominikKwietniak/}
\maketitle

\section{Introduction}





The theory of shadowing has been intensively developed in recent years. It is of extreme importance in the qualitative study of dynamical systems. Apart of the now standard shadowing property various variants of this concept were proposed. These new properties arise from modifications of the notion of pseudo-orbit together with different definitions of shadowing points. The main property that we consider in this paper is called \emph{two-sided limit shadowing}.
\begin{definitions}
We say that a double-infinite sequence $(x_i)_{i\in\Z}$ of points from a metric space $(X,d)$ is a \emph{two-sided limit pseudo-orbit} for a homeomorphism $f\colon X\to X$ if it satisfies $$d(f(x_i),x_{i+1})\rightarrow 0, \,\,\,\,\,\, |i|\rightarrow\infty.$$
A sequence $(x_i)_{i\in\Z}$ of points from $X$ is \emph{two-sided limit shadowed} if there exists $y\in X$ satisfying $$d(f^i(y),x_i)\rightarrow 0, \,\,\,\,\,\, |i|\rightarrow \infty.$$ In this situation we also say that $y$ \emph{two-sided limit shadows} $(x_i)_{i\in\Z}$ with respect to $f$. Finally, we say that $f$ has the \emph{two-sided limit shadowing property} if every two-sided limit pseudo-orbit is two-sided limit shadowed.
\end{definitions}
This notion was explored by the first author in \cite{C} where he proved that the two-sided limit shadowing property is related with a well known open problem asking whether every Anosov diffeomorphism is transitive. In fact, for Anosov diffeomorphisms on a differentiable manifold two-sided limit shadowing is equivalent to transitivity. In this paper we study this interesting property in detail. To this end we find it convenient to generalize this notion and introduce a strictly weaker property which we call \emph{two-sided limit shadowing with a gap}.
\begin{definitions}
We say that a sequence $(x_i)_{i\in\Z}$ is \emph{two-sided limit shadowed with gap $K\in\Z$ for $f$} if there exists
a point $y\in X$ satisfying
\begin{align*}
d(f^i(y),x_i)\to 0, \,\,\,\,\,\, i\to-\infty,\\
d(f^{K+i}(y),x_i)\to 0 \,\,\,\,\,\, i\to\infty.
\end{align*}
For $N\in\N_0$ we say that $f$ has the \emph{two-sided limit shadowing property with gap $N$} if every two-sided limit pseudo-orbit of $f$ is two sided limit shadowed with gap $K\in \Z$ with $|K|\le N$. We also say that $f$ has the \emph{two-sided limit shadowing property with a gap} if such an $N\in\N$ exists.
\end{definitions}

It is obvious that the two-sided limit shadowing property with gap $0$ is equivalent to the notion of two-sided limit shadowing defined above. It is also obvious that two-sided limit shadowing implies two-sided limit shadowing with a gap. Our results show that the two-sided limit shadowing property is the strongest among many existing variants of pseudo-orbit tracing properties. We prove that two-sided limit shadowing with a gap implies shadowing and transitivity, and gap $0$ implies even the specification property (Theorems \ref{thmA} and \ref{thmB} below). Combining this with results of \cite{KKO} we conclude that many shadowing properties follow from two-sided limit shadowing.

We show that the two-sided limit shadowing property with a gap may be strictly weaker than gap $0$. To this end we show that a two-point cycle is an example of a system which has the two-sided limit shadowing property with gap $1$ but do not have the two-sided limit shadowing property. Moreover, it follows from our results on expansive homeomorphism that for each $N\in \N$ there are examples of shifts of finite type with the two-sided limit shadowing property with gap $N$, which do not posses this property with a smaller gap.

Next we characterize the two-sided limit shadowing property with a gap for expansive homeomorphisms. We prove that for a transitive and expansive homeomorphism, shadowing, limit shadowing and two-sided limit shadowing with a gap are all equivalent (Theorem \ref{thmC}). It follows that for expansive homeomorphisms two-sided limit shadowing with a gap is equivalent to shadowing and transitivity, while gap $0$ is equivalent to topological mixing (Theorem \ref{thmD}).
As a corollary we extend the main result of \cite{C} to topologically Anosov homeomorphisms (Theorem \ref{generalized-Bernardo}). We also obtain a variant of Walters theorem from \cite{WaltersPOTP} (see also \cite[Theorem 2.3.18]{AH}) characterizing shifts of finite type as the shift spaces with shadowing. We show in Theorem \ref{thmE} that a shift space $X\subset A^\Z$ over a finite alphabet $A$ has the two-sided limit shadowing property with a gap if and only if $X$ is a transitive subshift of finite type. Moreover, gap $0$ is in this situation equivalent to topological mixing of $X$.
At the end we prove some lemmas which imply that there are systems with two-sided limit shadowing outside the class of expansive homeomorphisms.

\section{Preliminaries}

Through this paper $(X,d)$ is a compact metric space and $f\colon X\to X$ is a homeomorphism. By $\N$ we denote the set of positive integers and  $\N_0=\N\cup\{0\}$.
\begin{definitions}[Pseudo-orbits and shadowing]
Let $I\subset\Z$ be a nonempty set of consecutive integers. We say that a sequence $(x_i)_{i\in I}\subset X$ is a \emph{$\delta$-pseudo-orbit (for $f$)} if it satisfies $$d(f(x_i),x_{i+1})<\delta, \quad \text{for all } \,\,i \,\,\, \text{such that }i,i+1 \in I.$$
We call a pseudo-orbit \emph{finite} (\emph{positive, negative, two-sided}, respectively) if $I$ is finite ($I=\N_0$, $I=-\N_0$, $I=\Z$, respectively). A sequence $\{x_i\}_{i\in I}\subset X$ is $\eps$-\emph{shadowed (with respect to $f$)\footnote{From now on we omit references to $f$ when it is clear from the context which $f$ we have in mind.}} if there exists $y\in X$ satisfying $$d(f^i(y),x_i)<\eps, \quad \text{for all }i\in I.$$ We say that $f$ has the \emph{shadowing property} if for every $\eps>0$ there exists $\delta>0$ such that every two-sided $\delta$-pseudo-orbit is $\eps$-shadowed.
\end{definitions}

In our setting one may replace \emph{two-sided} by \emph{one-sided} or by \emph{finite} in the definition of shadowing without altering the notion. Excellent references for the shadowing property are Pilyugin's book \cite{P} and Aoki and Hiraide's monograph \cite{AH}. There are many variants of this property following the same schema: we define a notion of pseudo-orbit, then we introduce a concept of tracing, and we define a new notion of shadowing by demanding that each pseudo-orbit is traced in our new sense. We refer the reader to \cite{MK} for an unified approach to shadowing properties.

We will need the limit shadowing property introduced by Eirola, Nevanlinna and Pilyugin in \cite{ENP}.
\begin{definitions}[Limit shadowing]
We say that $(x_i)_{i\in\N_0}$ is a \emph{limit pseudo-orbit (for $f$)} if it satisfies $$d(f(x_i),x_{i+1})\rightarrow 0, \,\,\,\,\,\, i\rightarrow\infty.$$
A sequence $\{x_i\}_{i\in\N_0}$ is \emph{limit-shadowed} if there exists $y\in X$ such that $$d(f^i(y),x_i)\rightarrow 0, \,\,\,\,\,\, i\rightarrow \infty.$$
We say that $f$ has the \emph{limit shadowing property} if every limit pseudo-orbit is limit-shadowed.
\end{definitions}

The following lemma is obvious. We record it for further reference.
\begin{lemma}\label{2s-implies-lim}
If a homomorphism $f$ has the two-sided limit shadowing property with a gap, then $f$ and $f^{-1}$ have limit shadowing.
\end{lemma}
We will see later on that even if $f$ and $f^{-1}$ have limit shadowing, then $f$ need not to have two-sided limit shadowing.

\begin{definitions}[Topological transitivity and mixing]
We say that $f$ is \emph{transitive} if given any pair $(U,V)$ of nonempty open subsets of $X$ there exists $N\in\N$ such that $f^{N}(U)\cap V\neq\emptyset$. A map $f$ is \emph{topologically mixing} if for every pair $(U,V)$ of nonempty open subsets of $X$ there exists $N\in\N$ such that $f^n(U)\cap V\neq\emptyset$ for all $n\geq N$.
\end{definitions}
In our setting transitivity is equivalent to the existence of a point $x\in X$ such that its \emph{future orbit} $\{f^n(x);n\in\N\}$ is dense in $X$. It is proved in \cite{C} that an Anosov diffeomorphism is transitive if and only if it has the two-sided limit shadowing property.

\begin{definitions}[Specification]
Let $\tau=\{I_1,\dots,I_m\}$ be a finite collection of disjoint finite subsets of consecutive integers, $I_i=[a_i,b_i]\cap\Z$ for some $a_i,b_i\in\Z$, with
$$
a_1\le b_1 < a_2\le b_2 <\ldots < a_m\le b_m.$$
Let a map $P\colon \bigcup_{i=1}^mI_i\rightarrow X$ be such that for each $I\in \tau$ and $t_1, t_2\in I$ we have $$f^{t_2-t_1}(P(t_1))=P(t_2).$$
We call a pair $(\tau,P)$ a \emph{specification}. We say that the specification $S=(\tau,P)$ is \emph{$L$-spaced} if $a_{i+1}\geq b_i+L$ for all $i\in\{1,\dots,m-1\}$. Moreover, $S$ is \emph{$\eps$-shadowed} by $y\in X$ for $f$ if $$d(f^n(y),P(n))<\eps \,\,\,\,\,\, \textrm{for every} \,\,\, n\in \bigcup_{i=1}^mI_i.$$
We say that a homeomorphism $f\colon X\rightarrow X$ has the \emph{specification property} if for every $\eps>0$ there exists $L\in\N$ such that every $L$-spaced specification is $\eps$-shadowed. We say that $f$ has the \emph{periodic specification property} if for every $\eps>0$ there exists $L\in\N$ such that every $L$-spaced specification is $\eps$-shadowed by a periodic point $y$ such that $f^{b_m+L}(y)=y$. It is well known (see \cite{DGS}) that homeomorphisms with the specification property are topologically mixing.
\end{definitions}

\section{Consequences of the two-sided limit shadowing property}

We prove that the  two-sided limit shadowing property is the strongest one among many shadowing properties considered in the literature.
The first lemma is an easy consequence of the two-sided limit shadowing property with a gap and generalizes Lemma 2.4 in \cite{C}. For $x\in X$ we define the \emph{stable set of $x$} as the set of all points forward asymptotic to $x$, that is, $$W^s(x)=\{y\in X: \lim_{n\to\infty} d(f^n(y),f^n(x))=0\}.$$
Similarly, the \emph{unstable set of $x$} is the set $$W^u(x)=\{y\in X: \lim_{n\to\infty} d(f^{-n}(y),f^{-n}(x))=0\}.$$

\begin{lemma}\label{lem:2s-gap-connecting}
If a homeomorphism $f$ of a compact metric space $X$ has the two-sided limit shadowing property with gap $N$, then for every $x,y\in X$ there is an integer $K$, with $|K|\le N$, satisfying $$W^s(f^{-K}(x))\cap W^u(y)\neq\emptyset.$$
\end{lemma}
\begin{proof} Take any $x,y\in X$.
Consider the following sequence:
$$
z_n=\begin{cases}f^n(y),& \text{for }n<0,\\
f^n(x),& \text{for }n\geq 0.\end{cases}$$
It is obvious that $(z_n)_{n\in\Z}$ is a two-sided limit pseudo-orbit. Then there exist $\bar{z}\in X$ and $K\in\{-N,\ldots,N\}$ satisfying
$$
d(f^n(\bar{z}),f^n(y))\rightarrow 0, \quad n\rightarrow-\infty,
\text{ and \,\,}d(f^{K+n}(z),f^n(x))\rightarrow 0, \quad n\rightarrow\infty.
$$
This implies that $\bar{z}\in W^u(y)$ and $f^K(\bar{z})\in W^s(x)$. So $\bar{z}\in W^s(f^{-K}(x))\cap W^u(y)$.
\end{proof}

Now we prove that the two-sided limit shadowing property with a gap implies chain-transitivity. Recall that we say that $f$ is \emph{chain-transitive} if for every $\delta>0$ and every pair of points $(x,y)$ in $X\times X$ there are $n>0$ and a finite $\delta$-pseudo-orbit $(z_i)_{i=0}^n$ such that $z_0=x$ and $z_n=y$.

\begin{lemma}\label{lem:2s-gap-CT}
If a homeomorphism $f$ of a compact metric space $X$ has the two-sided limit shadowing property with gap $N$, then $f$ is chain-transitive.
\end{lemma}
\begin{proof}
Let $x,y\in X$. We want to find a $\delta$-pseudo-orbit starting at $x$ and ending at $y$ for any $\delta>0$. Consider the omega-limit set of $x$ and the alpha-limit set of $y$, denoted by $\omega(x)$ and $\alpha(y)$, respectively.
Choose $z_1\in\omega(x)$ and $z_2\in\alpha(y)$. Lemma \ref{lem:2s-gap-connecting} assures the existence of an integer $K\in\{-N,\ldots,N\}$ and a point $z\in W^u(z_1)\cap W^s(f^{-K}(z_2))$. Thus there exists $M\in\N$ such that $M>N$ and  $$d(f^{-M}(z),f^{-M}(z_1))<\frac{\delta}{2} \quad\textrm{and} \quad d(f^M(z),f^{M-K}(z_2))<\frac{\delta}{2}.$$
Since $\omega(x)$ and $\alpha(y)$ are compact and invariant subsets of $X$ we have $f^{-M}(z_1)\in\omega(x)$ and $f^{M-K}(z_2)\in\alpha(y)$. Then we can choose positive integers $M_1$ and $M_2$ such that $$d(f^{M_1}(x),f^{-M}(z_1))<\frac{\delta}{2} \quad\textrm{and} \quad d(f^{-M_2}(y),f^{M-K}(z_2))<\frac{\delta}{2}.$$
Thus $$d(f^{M_1}(x),f^{-M}(z))<\delta \quad\textrm{and} \quad d(f^{-M_2}(y),f^M(z))<\delta.$$
Now consider the following sequence 
$$x_n=\begin{cases}f^n(x), &\text{for }n=0,\dots,M_1-1,\\
f^{n-M_1}(f^{-N}(z)),& \text{for } n=M_1,\dots,M_1+2M-1,\\
f^{n-(M_1+2M)}(f^{-M_2}(y)), &\text{for } n=M_1+2M,\dots,M_1+2M+M_2.
\end{cases}
$$
Elements of this sequence are ordered as follows:
$$x,f(x),\ldots,f^{M_1-1}(x),f^{-M}(z),f^{-M+1}(z),\ldots,f^{M-1}(z),f^{-M_2}(y),\ldots,y.$$
It is clear that $(x_n)_{n=0}^{M_1+2N+M_2}$ is a finite $\delta$-pseudo-orbit connecting $x$ to $y$. Since this can be done for any $\delta>0$ we conclude the proof.
\end{proof}

The proof of the next lemma is an easy exercise, hence we omit it.
\begin{lemma}\label{lem:CT+sh=trans}
If a continuous map $f\colon X\to X$ on a compact metric space is chain-transitive and has the shadowing property, then $f$ is transitive.
\end{lemma}

In a joint work of Kulczycki, Oprocha and the second author of this paper \cite{KKO} it is proved that shadowing follows from chain-transitivity and the limit shadowing property. Following \cite{KKO} we present a simple proof of this fact for completeness.

\begin{theorem}[{\cite[Theorem 7.3]{KKO}}]\label{thm:CT+lim=sh}
If a continuous map $f\colon X\to X$ on a compact metric space is chain-transitive and has the limit shadowing property, then it has the shadowing property.
\end{theorem}
\begin{proof}
Aiming for a contradiction, suppose that $f$ does not have shadowing. Hence there is $\eps>0$ such that for any $n>0$ there is a finite $\frac{1}{n}$-pseudo-orbit $\alpha_n$ which cannot be $\eps$-shadowed by any
point in $X$. Using chain transitivity, for every $n$ there exists a $\frac{1}{n}$-pseudo-orbit $\beta_n$ such that the concatenated sequence
$\alpha_n\beta_n\alpha_{n+1}$ forms a finite $\frac{1}{n}$-pseudo-orbit. Then the infinite concatenation
$$
\alpha_1 \beta_1 \alpha_2 \beta_2 \alpha_3 \beta_3 \ldots
$$
is a limit pseudo orbit denoted by $(x_j)_{j=0}^\infty$. By limit shadowing, it is limit shadowed by some point $z\in X$. Hence, starting at some index $N>0$, the point $f^{N}(z)$ also $\eps$-shadows the limit pseudo orbit $(x_j)_{j=N}^\infty$. But this means that there is a finite pseudo-orbit $\alpha_n$ which is $\eps$-shadowed by some point of the form $f^{i}(z)$, which is a contradiction.
\end{proof}

Now we are ready to prove our first main result.

\begin{mainthm}\label{thmA}
If a homeomorphism $f$ of a compact metric space $X$ has the two-sided limit shadowing property with a gap, then it is transitive and has the shadowing property.
\end{mainthm}
\begin{proof}
By Lemma \ref{2s-implies-lim} $f$ has the limit shadowing property. Also, Lemma \ref{lem:2s-gap-CT} says that $f$ is chain-transitive. Then Theorem \ref{thm:CT+lim=sh} together with Lemma \ref{lem:CT+sh=trans} concludes the proof.
\end{proof}

Toward proving Theorem \ref{thmB} we recall some definitions and a result from \cite{KKO}.

\begin{definitions}[Average shadowing]
We say that $\{x_i\}_{i\in\N_0}$ is a \emph{$\delta$-average-pseudo-orbit} if there is an integer $N=N(\delta)\in\N$ such that for every $n\geq N$ and $k\geq0$ it holds $$\frac{1}{n}\sum_{i=0}^{n-1}d(f(x_{i+k})),x_{i+k+1})<\delta.$$
We say that $\{x_i\}_{i\in\N_0}$ is $\eps$-\emph{shadowed in average} if there exists $y\in X$ satisfying $$\limsup_{n\to\infty}\frac{1}{n}\sum_{i=0}^{n-1}d(f^i(y),x_i)<\eps.$$
We say that $f$ has the \emph{average shadowing property} if for every $\eps>0$ there exists $\delta>0$ such that every $\delta$-average-pseudo-orbit is $\eps$-shadowed in average. This property was introduced by Blank in \cite{B} and explored further by Zhang in \cite{Z}.
\end{definitions}

\begin{definitions}[Asymptotic average shadowing]
We say that $\{x_i\}_{i\in\N_0}$ is an \emph{asymptotic average-pseudo-orbit} if it satisfies $$\frac{1}{n}\sum_{i=0}^{n-1}d(f(x_i)),x_{i+1})\to0, \,\,\,\,\,\, n\to\infty.$$
We say that $\{x_i\}_{i\in\N_0}$ is \emph{asymptotically shadowed in average} if there exists $y\in X$ satisfying $$\frac{1}{n}\sum_{i=0}^{n-1}d(f^i(y),x_i)\to0, \,\,\,\,\,\, n\to\infty.$$
We say that $f$ has the \emph{asymptotic average shadowing property} if every asymptotic-average-pseudo-orbit is asymptotically shadowed in average. This property was introduced by Gu and Xia in \cite{GX}.
\end{definitions}


\begin{theorem}[{\cite[Theorem 3.8]{KKO}}]\label{thm:sh,tt,mix,..}
If $f\colon X\to X$ is a homeomorphism with the shadowing property then the following conditions are equivalent:
\begin{enumerate}
  \item $f$ is totally transitive, i.e., $f^n$ is transitive for every $n\in\N$;
  \item $f$ is topologically mixing;
  \item $f$ has the average shadowing property;
  \item $f$ has the asymptotic average shadowing property;
  \item $f$ has the specification property.
\end{enumerate}
\end{theorem}

Since two-sided limit shadowing implies shadowing and transitivity, above Theorem reduces the proof of Theorem \ref{thmB} to the following:

\begin{lemma}\label{lem:total-2s}
If a homeomorphism $f\colon X\to X$ has the two-sided limit shadowing property, then $f^n$ also has the two-sided limit shadowing property for every $n\in\Z\setminus\{0\}$.
\end{lemma}
\begin{proof}
Fix $n\in\Z\setminus\{0\}$. Let $(x_i)_{i\in\Z}$ be a two-sided limit pseudo-orbit for $f^n$. Then the sequence $(y_j)_{j\in\Z}$ defined by
$$
y_k=\begin{cases} x_i, &\text{if }k=in\text{ for some }i\in\Z,\\
f^{k-in}(x_i), &\text{if }in < k <(i+1)n\text{ for some }i\in\Z.
\end{cases}
$$
is a two-sided limit pseudo-orbit for $f$. Then it is two-sided limit shadowed for $f$ by a point $z\in X$. Hence the sequence $(x_i)_{i\in\Z}$ is two-sided limit shadowed by the $f^n$-orbit of $z$. This concludes the proof.
\end{proof}

\begin{mainthm}\label{thmB}
If a homeomorphism of a compact metric space has the two-sided limit shadowing property then it is topologically mixing and has specification, average shadowing and asymptotic average shadowing.
\end{mainthm}
\begin{proof}
Since two-sided limit shadowing implies transitivity, Lemma \ref{lem:total-2s} implies that $f$ is totally transitive. Since it also implies shadowing, Theorem \ref{thm:sh,tt,mix,..} concludes the proof.
\end{proof}


As a consequence of Theorem \ref{thmB} we know that any homeomorphism with the two-sided limit shadowing property has all features of systems with the specification property, and these are numerous (see Chapter 21 of \cite{DGS} for details). For example by Proposition 21.6 in \cite{DGS} we obtain that $f$ has positive topological entropy. We note that this also implies that there are no homeomorphisms with the two-sided limit shadowing property on the unit interval or the unit circle since there are no homeomorphism with positive topological entropy on them (see \cite{W} Lemmas 7.14 and 7.14.1).

Finally we exhibit a simple example of a homeomorphism on a compact metric space which has the two-sided limit shadowing property with a gap but does not have the two-sided limit shadowing property.

\begin{example}
Consider $X=\{a,b\}$ and define $f\colon X\to X$ by $f(a)=b$ and $f(b)=a$. It is easy to see that $f$ is a homeomorphism which is not topologically mixing, so Theorem \ref{thmB} shows that $f$ does not have the two-sided limit shadowing property. We claim that $f$ has the two-sided limit shadowing property with gap $1$. For each limit pseudo-orbit $(x_n)_{n\in\N}$ there exists $N\in\N$ such that $x_{N+k}=a$ if $k$ is even and $x_{N+k}=b$ if $k$ is odd. So for each two-sided limit pseudo-orbit $(x_n)_{n\in\Z}$ there exist $N_1,N_2\in\N$ such that $x_{N_1+k}=a$, $x_{-N_2-k}=a$ if $k$ is even and $x_{N_1+k}=b$, $x_{-N_2-k}=b$ if $k$ is odd. It follows that if $N_1=N_2\mod2$ then either $a$ or $b$ two-sided limit shadows $(x_n)_{n\in\Z}$. Otherwise, (if $N_1=N_2+1\mod2$) neither $a$, nor $b$ two-sided limit shadows $(x_n)_{n\in\Z}$, but it is easy to check that one of them two-sided limit shadows $(x_n)_{n\in\Z}$ with gap $1$. Hence $f$ has the two-sided limit shadowing with gap $1$.
\end{example}

The above example also shows that topologically mixing is not necessary for the two-sided limit shadowing property with a gap, and the same holds for the average shadowing, asymptotic average shadowing and specification properties. We proved that  $f$ as above has the two-sided limit shadowing property with gap $1$, but obviously $f^2$ does not, so an analog of Lemma \ref{lem:total-2s} for two-sided limit shadowing with a gap is not true. In view of all results of this section the following question seems natural:

\begin{q}\label{question}
Does every homeomorphism with shadowing and specification properties have the two-sided limit shadowing property?
\end{q}


\section{Two-sided limit shadowing for expansive homeomorphisms}

Recall that a homeomorphism $f\colon X\to X$ is \emph{expansive} if there exists a constant $\eps>0$ (called \emph{expansivity constant}) such that if $x,y\in X$ satisfy $d(f^n(x),f^n(y))\leq\eps$ for every $n\in\Z$, then $x=y$. The first author of this paper in \cite[Lemma 2.2]{C} answered question \ref{question} assuming that $f$ is expansive:

\begin{lemma}[\cite{C}]\label{lem:Bernardo} Every expansive homeomorphism $f\colon X\to X$ with the shadowing and specification properties has the two-sided limit shadowing property.
\end{lemma}

In this section we extend Lemma \ref{lem:Bernardo}, and at the same time we extend Corollary 7.5 of \cite{KKO} (in case of homeomorphisms) and the main Theorem of \cite{LS} (note that the definition of \emph{limit shadowing} in the former is different than here, but these notions are equivalent thanks to \cite[Theorem 7.5]{KKO}, where limit shadowing of  \cite{LS} appears as \emph{s-limit shadowing}).

\begin{mainthm}\label{thmC}
Let $f$ be a transitive and expansive homeomorphism of a compact metric space $X$. Then the following are equivalent:
\begin{enumerate}
  \item $f$ has the shadowing property
  \item $f$ has the limit shadowing property
  \item $f$ has the two-sided limit shadowing property with a gap.
\end{enumerate}
\end{mainthm}
\begin{proof}
The equivalence of $(1)$ and $(2)$ follows from Corollary 7.5 of \cite{KKO}. The implication from $(3)$ to $(1)$ is proved in Theorem \ref{thmA}. It remains to prove that shadowing and transitivity implies two-sided limit shadowing with a gap for expansive homeomorphisms. It follows from the Spectral Decomposition Theorem (\cite[Theorem 3.1.11]{AH}) for topologically Anosov maps that there exists an integer $N$ such that $X$ can be written as a disjoint union, $X=B_0\cup\ldots\cup B_{N-1}$ of non-empty closed sets satisfying
$$
f(B_i)=B_{(i+1)\bmod N} \quad\text{for }i=0,\ldots,N-1,
$$
and such that $f^N|_{B_i}\colon B_i \to B_i$ is topologically mixing (hence has the specification property) for each $i$. Note that this happens because $f$ is transitive, so it has only one chain-recurrent class. We claim that $f$ has the two-sided limit shadowing property with gap $N-1$. Let $(x_n)_{n\in\Z}$ be a two-sided limit pseudo-orbit for $f$. It is easy to see that there exists positive integers $L_1,L_2\in\N$ such that
$$
x_{L_1+k}\in B_{k\bmod N}\text{ and } x_{-L_2-k}\in B_{(-k)\bmod N}\quad\text{for each }k\ge 0.
$$
Let $K=(L_1+L_2)\bmod{N}$.
Take any $n_0\in \N$ such that $(-L_2)\bmod{N}+Nn_0\ge L_1$. It follows that
$$
f^K(x_{(-L_2)\bmod N+Nn})\in B_0\quad\text{for }n> n_0.
$$
Let $p_1,\ldots,p_{n_0}$ be any points in $B_0$.
Then the sequence $(y_n)_{n\in\Z}$ defined by
$$
y_n=\begin{cases}
x_{(-L_2\bmod N)+Nn},&\text{for }n\le 0,\\
p_n,&\text{for }1\le n\le n_0,\\
f^{K}(x_{(-L_2)\bmod N+Nn})&\text{for }n>n_0,
\end{cases}
$$
is a two-sided limit pseudo-orbit for $f^N$ contained in $B_0$. By Lemma \ref{lem:Bernardo} $f^N$ restricted to $B_0$ has the two-sided limit shadowing property, since $f^N|_{B_0}$ has both shadowing and specification. Therefore $(y_n)_{n\in\Z}$ is two-sided limit shadowed for $f^N$ by a point $z\in B_0$. It implies that
$f^{-K}(z)$ limit shadows for $f^N$ the limit pseudo-orbit $(f^{-K}(y_n))_{n=0}^\infty$ (this is a limit pseudo-orbit with respect to $f^N$).
For $n>n_0$ we have $x_{(-L_2)\bmod N+Nn}=f^{-K}(y_n)$. Hence the sequence $(x_n)_{n\in\Z}$ is two-sided limit shadowed for $f$ with gap $K$ by the point $z\in X$ since uniform continuity of $f$ and two-sided limit shadowing for $f^N$ clearly imply
\begin{align*}
d(f^i(z),x_i)\to 0, \,\,\,\,\,\, i\to-\infty,\\
d(f^{-K+i}(z),x_i)\to 0 \,\,\,\,\,\, i\to\infty.
\end{align*}
This concludes the proof.
\end{proof}

Now we characterize the two-sided limit shadowing property and the two-sided limit shadowing property with a gap for expansive homeomorphisms:

\begin{mainthm}\label{thmD}
Let $f$ be an expansive homeomorphism of a compact metric space $X$. Then
\begin{enumerate}
  \item $f$ is transitive and has the shadowing property if and only if $f$ has the two-sided limit shadowing property with a gap;
  \item $f$ is topologically mixing and has the shadowing property if and only if $f$ has the two-sided limit shadowing property.
\end{enumerate}
\end{mainthm}
\begin{proof}
By Theorem \ref{thmA} we know that the two-sided limit shadowing property with a gap implies shadowing and transitivity. The converse under additional assumption of expansivity follows from Theorem \ref{thmC}. This concludes the proof of the first equivalence. The second equivalence follows from Lemma \ref{lem:Bernardo} (noting that topologically mixing expansive homeomorphisms with shadowing have specification, see \cite[Theorem 11.5.13]{AH}) and from Theorem \ref{thmB}.
\end{proof}

By the Spectral Decomposition Theorem (\cite[Theorem 3.1.11]{AH}), an expansive homeomorphism of a compact and connected metric space with the shadowing property is topologically mixing if and only if it is transitive. Hence we have the following corollary.

\begin{theorem}\label{generalized-Bernardo}
If $f$ is an expansive homeomorphism of a compact and connected metric space $X$, then the following are equivalent:
\begin{enumerate}
  \item $f$ is transitive and has the shadowing property;
  \item $f$ is topologically mixing and has the shadowing property;
  \item $f$ has the two-sided limit shadowing property with a gap;
  \item $f$ has the two-sided limit shadowing property.
\end{enumerate}
\end{theorem}

A well known result of Walters \cite{WaltersPOTP} (see also \cite[Theorem 2.3.18]{AH}) characterizes invertible shifts of finite type as the shift space with the shadowing property. Using this theorem and our Theorem \ref{thmD} we obtain characterizations of transitive and topologically mixing shifts of finite type in terms of two-sided limit shadowing properties.

\begin{mainthm}\label{thmE}
Let $X$ be a shift space. Then
\begin{enumerate}
  \item $X$ is a transitive shift of finite type if and only if $\sigma\colon X\to X$ has the two-sided limit shadowing property with a gap;
  \item $X$ is topologically mixing shift of finite type if and only if $\sigma\colon X\to X$ has the two-sided limit shadowing property.
\end{enumerate}
\end{mainthm}

\section{Examples}

In this section we prove two results which will allow us to provide examples of homeomorphisms with the two-sided shadowing property
\begin{enumerate}
  \item which are non-expansive;
  \item without the periodic specification property.
\end{enumerate}

\begin{definitions}[Shift spaces]
For any compact metric space $(X,d)$ we consider $X^{\Z}$ the product of countable many copies of $X$ with the Tichonov (product) topology. Points in $X^{\Z}$ are sequences $(x_n)_{n\in\Z}$ whose all coordinates belong to $X$. We define a map $\sigma:X^{\Z}\to X^{\Z}$ by $$\sigma((x_n)_{n\in\Z})=(x_{n+1})_{n\in\Z}.$$ This map is called the \emph{shift map}. It is easily seen that $\sigma$ is a homeomorphism of a compact space $X^{\Z}$. A special case of this construction is a \emph{full $r$-shift}  $\Omega_r=X^\Z$, where $r\in\N$ and $X=\Lambda_r=\{0,1,\ldots,r-1\}$ is equipped with discrete metric. Any $\sigma$-invariant and closed subset of $\Omega_r$ is a \emph{shift space}. A shift space $Z\subset \Omega_r$ is a \emph{subshift of finite type} if there exists $N$ and a subset $\mathcal{L}$ of $\prod_{i=0}^N\Lambda_r$ such that $x=(x_n)_{n\in\Z}\in Z$ if and only if $x_jx_{j+1}\ldots x_{j+N}\in \mathcal{L}$ for every $j\in\Z$.
\end{definitions}

We prove the following:
\begin{theorem}\label{shift}
For every compact metric space $X$ the shift map $\sigma\colon X^\Z\to X^\Z$ has the two-sided limit shadowing property.
\end{theorem}
\begin{proof}
As two sided limit shadowing does not depend on the choice of metric for the space we can pick any equivalent metric for $X^\Z$.
We equip $X^\Z$ with the metric $D$ defined for $x=(x_{j})_{j\in\Z}, y=(y_{j})_{j\in\Z}\in X^\Z$ by
\[
D(x,y)=\sup_{j\in\Z} \frac{d(x^{(j)},y^{(j)})}{2^{|j|}},
\]
where $d$ is any metric for $X$ such that $\diam X \le 1$.
Let $\{x^{(n)}\}_{n\in\Z}$ be any two-sided limit pseudo-orbit for $\sigma$. It follows that for every $p\ge 1$ there is $N_p$ such that
$D(\sigma(x^{(m)}),x^{(m+1)})\le 2^{-p-1}$ for all $|m|\ge N_p$. It implies that for $m$ as before and for each $j\in\Z$ we have
\begin{equation}\label{ineq:2k}
\frac{1}{2^{p+1}}\ge D(\sigma(x^{(m)}),x^{(m+1)})\ge \frac{1}{2^{|j|}} d(x^{(m)}_{j+1},x^{(m+1)}_j).
\end{equation}
Define a point $x\in X^\Z$ by $x_j=x^{(j)}_0$ for $j\in\Z$, in other words
\[x=(\ldots,x^{(-1)}_0,x^{(0)}_0,x^{(1)}_0\ldots)\]
We claim that $x$ two-sided asymptotically traces $\{x^{(n)}\}_{n\in\Z}$. To see this, note first that
\[
\sigma^m(x)_k=x_0^{(m+k)}\quad\text{for all }k,m\in\Z.
\]
On the other hand
\[
D(\sigma^m(x),x^{(m)})=\sup_{j\in\Z}\frac{1}{2^{|j|}}d(\sigma^m(x)_j,x^{(m)}_j)=\sup_{j\in\Z}\frac{1}{2^{|j|}}d(x_0^{(m+j)},x^{(m)}_j).
\]
By the triangle inequality for all $m\in\Z$ and $k$ positive
\begin{multline*}
d(x_0^{(m+k)},x^{(m)}_k)\le \sum_{j=0}^{k-1}d(x_j^{(m+k-j)},x^{(m+k-j-1)}_{j+1})=
\\=d(x_0^{(m+k)},x^{(m+k-1)}_1)+d(x_1^{(m+k-1)},x^{(m+k-2)}_2)+\ldots+d(x_{k-1}^{(m+1)},x^{(m)}_k).
\end{multline*}
Similarly, for $k$ negative
\begin{multline*}
d(x_0^{(m+k)},x^{(m)}_k)\le \sum_{j=0}^{|k|-1}d(x_{-j}^{(m+k+j)},x^{(m+k+j+1)}_{-j-1})=
\\=d(x_0^{(m+k)},x^{(m+k+1)}_{-1})+d(x_{-1}^{(m+k+1)},x^{(m+k+2)}_{-2})+\ldots+d(x_{k+1}^{(m-1)},x^{(m)}_k).
\end{multline*}

If $m> N_p$, $k>0$ and $0\le j <k$, then we may apply \eqref{ineq:2k} and obtain
\[
d(x_j^{(m+k-j)},x^{(m+k-j-1)}_{j+1})\le \frac{2^{|j|}}{2^{p+1}}.
\]
In particular,
\[
\sum_{j=0}^{k-1}d(x_j^{(m+k-j)},x^{(m+k-j-1)}_{j+1})\le \frac{1+2+\ldots+2^{k-1}}{2^{p+1}}\le \frac{2^{k}}{2^{p+1}}
\]
If $m<-N_p$, $k<0$ and $0\le j \le |k|-1$, then we may apply \eqref{ineq:2k} and obtain
\[
d(x_{-j}^{(m+k+j)},x^{(m+k+j+1)}_{-j-1})\le\frac{2^{j+1}}{2^{p+1}}.
\]
In particular,
\[
\sum_{j=0}^{|k|-1}d(x_{-j}^{(m+k+j)},x^{(m+k+j+1)}_{-j-1})\le \frac{2+2^2+\ldots+2^{|k|}}{2^{p+1}}\le \frac{2^{|k|}}{2^{p}}.
\]
Hence $D(\sigma^m(x),x^{(m)})\le 2^{-p}$ for $|m|\ge N_p$, and the proof is finished. \end{proof}

It follows from theorem above and Theorem \ref{thmB} that for every compact metric space $X$ the shift map $\sigma\colon X^\Z\to X^\Z$ has specification and shadowing. But if the topological dimension of $X$ is non-zero, then no homeomorphism of $X^\Z$ can be expansive by \cite[Theorem 2.3.13]{AH} (this is a corollary to a theorem of Ma\~{n}e \cite{Mane} (see also \cite[Theorem 2.2.40]{AH}), which states that if $f\colon X\to X$ is an expansive homeomorphism of a compact metric space, then the topological dimension of $X$ is finite). Hence we obtain the following:

\begin{corollary}
There exists a non-expansive homeomorphism with the two-sided limit shadowing property.
\end{corollary}

\begin{center}
\begin{figure}[ht]\label{fig}
\includegraphics{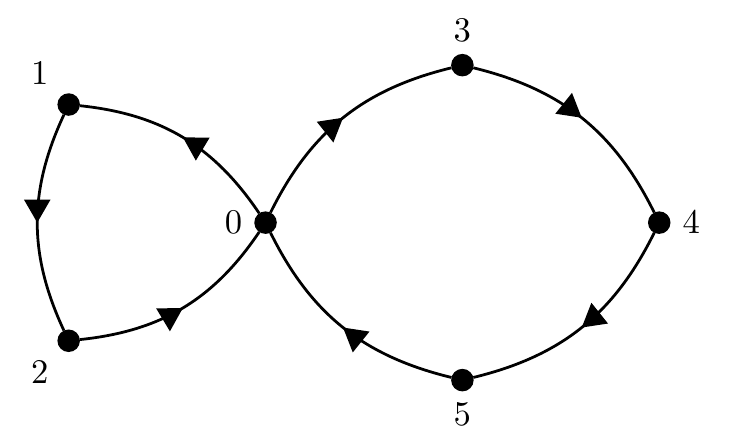}
\caption{A graph presenting the shift space $X_{(3,4)}$, given by the set $\mathcal{L}=\{01,12,20,03,34,45,50\}$.}
\end{figure}
\end{center}

\begin{theorem}\label{product}
Assume that for each $n\in \N$ we have a compact metric space $X_n$ and a homeomorphism $f_n\colon X_n\to X_n$ with the two-sided limit shadowing property. Let $X=\prod_{n\in\N} X_n$ be the product metric space and $F\colon X\to X$ be the homeomorphism given by
$$
F(x_1,x_2,x_3,\ldots)=(f_1(x_1),f_2(x_2),f_3(x_3),\ldots).
$$
Then $F$ has the two-sided limit shadowing property.
\end{theorem}
\begin{proof}
We equip $X$ with the metric $D$ defined for $x=(x_n)_{n\in\N}, y=(y_n)_{n\in\N}\in X$ by
\[
D(x,y)=\sup_{n\in\N} \frac{d_n(x_n,y_n)}{2^n},
\]
where $d_n$ is any metric for $X_n$ such that the diameter of $X_n$ with respect to $d_n$ fulfills $\diam X_n \le 1$.
It is then easy to see that if $(x^{(j)})_{j\in\Z}$ is a two-sided limit pseudo-orbit for $F$, then for each
$n\in\N$ the sequence $(x^{(j)}_n)_{j\in\Z}$ is a two-sided limit pseudo-orbit for $f_n$. Let $y_n$ be a point which two-sided limit shadows
$(x^{(j)}_n)_{j\in\Z}$ with respect to $f_n$. It is again easy to see that $(y_1,y_2,y_3,\ldots)$ two-sided limit shadows $(x^{(j)})_{j\in\Z}$ for $F$.
\end{proof}

Let $p$ and $q$ be relatively prime integers. Set $r=p+q-1$. Define a shift of finite type $X_{(p,q)}\subset\Omega_r$ by specifying
\begin{multline*}
\mathcal{L}=\{01,12,\ldots,(p-2)(p-1),(p-1)0,\\ 0p,p(p+1),\ldots,(p+q-2)(p+q-1),(p+q-1)0\}.
\end{multline*}
In other words, $X_{(p,q)}$ consists of sequences of vertices visited during a bi-infinite walk on the directed graph with two loops: one of length $p$ with vertices labelled $0,\ldots,p-1$ and one of length $q$ with vertices labelled $0,p,p+1,\ldots,p+q-2$.
Since the graph is connected and has two cycles with relatively prime lengths it presents a topologically mixing shift of finite type (for details see \cite{LM}). Moreover, this shift of finite type does not have any periodic point with primary period smaller than $\min\{p,q\}$. Let $(p_j)_{j=1}^\infty$ be a strictly increasing sequence of prime numbers. Let $X_n=X_{(p_n,p_{n+1})}$ for $n\in\N$, and $\sigma_n$ be a shift transformation on $\Omega_{p_{n}+p_{n+1}-1}$ restricted to $X_n$. This family fulfills the assumptions of Theorem \ref{product} so the product system has the two-sided limit shadowing property. It is easy to see that the product system $F=\sigma_1\times\sigma_2\times\ldots$ on $X=\prod_{n=1}^\infty X_n$ also has the specification property, but this specification property is not periodic, since there are no periodic points for $F$ in $X$. Hence we obtain our last result.

\begin{corollary}
There exists a homeomorphism with the two-sided limit shadowing property but without the periodic specification property.
\end{corollary}

\section*{Acknowledgements}
The first named author thanks Welington Cordeiro for helpful discussions during the preparation of this paper. The second named author is grateful to Jakub Byszewski for fruitful discussions and to Piotr Oprocha for sharing an idea leading to a simplified presentation of examples in the last section.

\end{document}